\documentclass[11pt]{amsart}

\usepackage[colorlinks=true, linkcolor=blue, anchorcolor=blue, citecolor=blue, filecolor=blue, menucolor= blue, urlcolor=red]{hyperref}

\usepackage{amsmath,amssymb,amsthm}
\usepackage{verbatim}
\usepackage{comment}

\long\def\eatit#1{}
\newtheorem{thm}{Theorem}[section]
\newtheorem{prop}[thm]{Proposition}
\newtheorem{lem}[thm]{Lemma}
\newtheorem{cor}[thm]{Corollary}

\newtheorem{conj}[thm]{Conjecture}

\theoremstyle{definition}

\newtheorem{ques}[thm]{Question}
\newtheorem{rem}[thm]{Remark}

\newcommand{\pr}[1]{{{\bf P}^{#1}}}

\begin{document}
\title{Containment results for ideals of various configurations of points in $\pr N$}


\author{Cristiano Bocci}
\address{Dipartimento di Ingegneria dell'Informazione e Scienze Matematiche\\
Universit\`a degli Studi di Siena\\
Pian dei mantellini, 44\\
53100 Siena, Italy}
\email{cristiano.bocci@unisi.it}
\author{Susan Cooper}
\address{Department of Mathematics\\
Central Michigan University\\
Mount Pleasant, Mich. 48859 USA}
\email{s.cooper@cmich.edu}
\author{Brian Harbourne}
\address{Department of Mathematics\\
University of Nebraska\\
Lincoln, NE 68588-0130 USA}
\email{bharbour@math.unl.edu}

\thanks{The first author was partially supported by GNSAGA of INdAM (Italy). The third 
author's work on this project
was sponsored by the National Security Agency under Grant/Cooperative
agreement ``Advances on Fat Points and Symbolic Powers,'' Number H98230-11-1-0139.
The United States Government is authorized to reproduce and distribute reprints
notwithstanding any copyright notice. All three authors thank the referee for his 
careful reading of this paper and for his
helpful comments.}

\begin{abstract}
Guided by evidence coming from a few key examples and attempting to unify previous work
of Chudnovsky, Esnault-Viehweg, Eisenbud-Mazur, Ein-Lazarsfeld-Smith, 
Hochster-Huneke and Bocci-Harbourne, 
Harbourne and Huneke recently formulated a series of conjectures 
that relate symbolic and regular powers of
ideals of fat points in $\pr N$. In this paper we propose another conjecture
along the same lines (Conjecture 3.9), and we verify it and the
conjectures of Harbourne and Huneke for a variety of configurations of points.  
\end{abstract}

\date{February 17, 2013}



\maketitle

\section{Introduction}

\subsection{Historical Overview}
The difference between ordinary and symbolic powers of ideals underlies many fundamental problems
in algebraic geometry and commutative algebra. A manifestation of these differences in algebraic 
geometry is the occurrence of non-linearly normal embeddings of varieties $V$ in projective space,
or, more generally, the typical lack of surjectivity of canonical maps of the form
$H^0(V, {\mathcal F})^{\otimes r}\to H^0(V, {\mathcal F}^{\otimes r})$,
given a sheaf of modules ${\mathcal F}$ on $V$. In commutative algebra
these differences are related to the occurrence of associated primes
for powers $I^r$ of an ideal which are not associated primes for $I$ itself. 

One of the simplest contexts of interest in this regard is that of ideals of points in projective space.
So consider the ideal $I$ of a finite set of points $P_1,\ldots,P_n\in\pr N$.
Thus $I$ is the radical ideal $I=\cap_j I(P_j)$ in the polynomial ring $K[\pr N]=K[x_0,\ldots,x_N]$
over the ground field $K$, where $I(P_j)$ is the ideal generated by all homogeneous
polynomials (i.e., by all forms) which vanish at each point $P_j$.
The symbolic power $I^{(m)}$ of $I$ in this case is $\cap_j (I(P_j)^m)$.
More generally, if $I\subset K[\pr N]$ is any homogeneous ideal,
then the associated primes $P_i$ for $I$ are homogeneous
and we have a primary decomposition $I=\cap_i Q_i$ where each
$Q_i$ is homogeneous and $P_i$-primary. Let $P_j'$ be the associated primes
for $I^m$ and let $I^m=\cap_jQ_j'$ be a primary decomposition such that
each $Q_j'$ is homogeneous and $P_j'$-primary. Then the $m$-th symbolic power
is $I^{(m)}=\cap_{j:P_j'\subseteq P_i \text{ for some }i}Q_j'$.

For simplicity, we will assume $K$ is algebraically closed.
In some cases we will assume $K$ has characteristic 0, but only where we 
say this explicitly.

Given a homogeneous ideal $(0)\ne J\subseteq K[\pr N]$, $\alpha(J)$
denotes the least degree of a non-zero form in $J$.
It is easy to see that $\alpha(J^r)=r\alpha(J)$, 
but the behavior of $\alpha(J^{(r)})$ is much more complicated and not well-understood. 
For an ideal $I$ of a finite set of points of $\pr N$ with $K=\mathbb{C}$, 
Skoda \cite{refSk}, in work on complex functions with applications
to number theory, sharpened a result of Waldschmidt \cite{refW}
by showing $\alpha(I^{(m)})/m\geq \alpha(I)/N$ for all $m\geq 1$. 
A further refinement, $\alpha(I^{(m)})/m\geq \alpha(I^{(n)})/(n+N-1)$, is given in \cite[Lemme 7.5.2]{refW2}.
Chudnovsky \cite{refCh} improved the original Waldschmidt-Skoda bound when $N=2$ by showing
$\alpha(I^{(m)})/m\geq( \alpha(I)+1)/2$ for all $m\geq 1$ (over any field)
and conjectured for any $N$ that $\alpha(I^{(m)})/m\geq( \alpha(I)+N-1)/N$.
Esnault-Viehweg \cite{refEV} (using methods of complex algebraic geometry
such as vanishing theorems) made partial progress towards
these conjectures by showing $\alpha(I^{(m)})/m\geq( \alpha(I^{(n)})+1)/(n+N-1)$ for $m,n\geq1$.

On a different but actually closely related tack,
Ein-Lazarsfeld-Smith \cite{refELS} (using multiplier ideals) and Hochster-Huneke \cite{refHoHu}
(using tight closure) showed that $I^{(rN)}\subseteq I^r$ for all $r>0$
(as one case of more general results).
This raises the question of what the smallest constant $c$ is such that $I^{(m)}\subseteq I^r$ whenever
$m > cr$. The first and third authors \cite{refBH2} showed in fact that $c=N$ is optimal
(in the sense that for each $c<N$, there is an ideal $I$ of points  in $\pr N$ such that 
$I^{(m)}\subseteq I^r$ fails for some $m$ and $r$ with $m>cr$).
The third author (see \cite{refB. et al}), following up on questions of Huneke 
and with the goal of obtaining tighter containments, showed for some ideals $I$ that
$I^{(Nr-(N-1))}\subseteq I^r$ holds for all $r>0$, and
conjectured that this holds for all $I$.
Motivated by this, by the third author's observation that
$I^{(rN)}\subseteq I^r$ implies Skoda's bound in a characteristic free way
(see \cite{refHaHu}; also see the discussion
in \cite{refSc}---the latter paper also obtains Skoda's result for positive characteristics, 
using methods growing out of tight closure) and by the 
Eisenbud-Mazur \cite{refEM} conjecture on evolutions,
the third author and Huneke \cite{refHaHu} formulated additional conjectures,
refining previous conjectures and tightening them further by considering
containments of $I^{(m)}$ in products of the form $M^jI^r$, where
$M\subset K[\pr N]$ is the ideal generated by the variables.

Other than theoretical considerations, these new conjectures are based on only a
few key examples. The goal of the present paper is to collect together what is known,
and to broaden the base of support
of these conjectures by proving additional cases of the conjectures.
We also propose a new conjecture, Conjecture \ref{ourconj}, along the same lines as those
of \cite{refHaHu}, and we verify this new conjecture in a range of cases.

\subsection{Technical Overview}
Although questions of containments of symbolic powers in ordinary powers
is of interest in general (and there is some evidence that the conjectures of
\cite{refHaHu} hold more generally, not just for ideals of points),
symbolic powers of ideals of {\it fat points} are of special interest,
both for their conceptual simplicity and as a starting point for trying to understand
these containment problems. To recall, given a finite set of distinct 
points $P_1, \ldots, P_n$ in $K[\pr N]$ and non-negative integers $m_1, \ldots, m_n$, a fat 
point subscheme is the subscheme defined by an ideal of the form
$I = I(P_1)^{m_1} \cap I(P_2)^{m_2} \cap \cdots \cap I(P_n)^{m_n}$, 
where $I(P_i)$ is the ideal generated by the forms that vanish at $P_i$.  The $m$th symbolic 
power of such an ideal turns out to be 
$I^{(m)} = I(P_1)^{mm_1} \cap I(P_2)^{mm_2} \cap \cdots \cap I(P_n)^{mm_n}$.

In the paper \cite{refHaHu}, Harbourne and Huneke consider the following general questions
(among others):

\begin{ques}\cite[Questions 1.3, 1.4 and Conjecture 4.1.5]{refHaHu}\label{q}
Let $R = K[\pr N]$ and $M = (x_0, \ldots, x_N)$ be the maximal homogeneous ideal of $R$.
Let $I \subseteq R$ be a homogeneous ideal.  
\begin{enumerate}
\item For which $m, i$ and $j$ do 
we have $I^{(m)} \subseteq M^jI^i$?
\item For which $j$ does $I^{(rN)} \subseteq M^jI^r$ hold for all homogeneous 
ideals $I \subseteq R$ and all $r$?
\item For which $j$ does $I^{(rN-N+1)} \subseteq M^jI^r$ hold, given that
$I^{(rN-N+1)} \subseteq I^r$ holds for all $r$?
\end{enumerate}
\end{ques} 

The first question is a natural outgrowth of the Eisenbud-Mazur conjecture
(which concerns containment of symbolic squares $P^{(2)}$ of prime ideals $P$
in $MP$). Given that it is known that $I^{(rN)} \subseteq I^r$ holds for all $r>0$,
the second question arises naturally if one tries to decrease the gap
between $I^{(rN)}$ and $I^r$. Another way to decrease the gap is by making
the exponent of the symbolic power smaller. As discussed above, one cannot in general
do this by making the coefficient $N$ of $r$ smaller. This suggested subtracting 
something off, which led to the conjecture that $I^{(rN-(N-1))} \subseteq I^r$ 
\cite[Conjecture 8.4.2]{refB. et al}, and
given this conjecture it is natural to ask if one can decrease the gap further. 
This leads to the third question.

In the spirit of Question \ref{q}, Harbourne and Huneke state a series of conjectures 
(see Section \ref{theconjectures}) involving containment of 
symbolic powers of ideals in their regular powers,
as well as bounding the initial degrees of symbolic powers in terms of the initial degrees
of the ideal itself.  We consider these conjectures specialized to 
various configurations of points.  In Section \ref{prelim} we recall some facts
and prove a few others that will be useful later on.   
In Section \ref{theconjectures} we state the conjectures of interest.
In Section \ref{symisord} we verify that the conjectures hold  
under the assumption that the symbolic and ordinary powers are the same, such as 
when the points comprise a complete intersection. In Section \ref{oddsmoothconic}
we consider the case of points on smooth plane conics. In Section \ref{stars}
we study certain important special point sets of $\pr N$ called star configurations.
In Section \ref{HyperplaneSection}, we look at
sets of points contained in a hyperplane, as a corollary of which we recover and extend
a result of \cite{refDu}.
In Section \ref{general} we investigate general sets of points in the plane.  
Finally, we conclude in Section \ref{additional} 
with a few additional characteristic 0 results for $\pr N$.

\section{Preliminaries}\label{prelim}

Given a homogeneous ideal $0\ne I\subseteq R=K[{\bf P}^N]$, let 
$\alpha(I)$ be the least degree $t$ such that the homogeneous
component $I_t$ in degree $t$ is not zero. 
Thus $\alpha$ is, so to speak, the degree
in which the ideal begins (i.e., the degree of a generator of least degree).
Throughout the paper, since it will always be clear from context what the ring $K[{\bf P}^N]$ is,
we will simply use $M$ to denote the maximal proper homogeneous ideal of $K[{\bf P}^N]$
(i.e., the irrelevant ideal $M=(x_0,\ldots,x_N)\subseteq K[{\bf P}^N]=K[x_0,\ldots,x_N]$).
Thus $\alpha(I)$ is also the $M$-adic order of $I$ (i.e., the largest $t$ such that
$I\subseteq M^t$). 

The Hilbert function of $I$ is the function $h_I(t)=\dim_K I_t$ for $t\geq 0$.
For $t\gg0$, $h_I$ is a polynomial.
Let $\tau(I)$ be the least degree $t$ such that
the Hilbert function  becomes equal to the Hilbert polynomial of $I$,
let $\sigma(I)=\tau(I)+1$ and let $0\to F_N\to\cdots\to F_0\to I\to0$
be a minimal free resolution 
of $I$ over $R$, where $F_i$ as a graded $R$-module is
$\oplus R[-b_{ij}]$. Then the {\it Castelnuovo-Mumford regularity} $\operatorname{reg}(I)$
of $I$ is the maximum over all $i$ and $j$ of $b_{ij}-i$.

We say that $I$ is saturated if $M$ is not an associated prime of $I$.
The saturation $\operatorname{sat}(I)$ of $I$ is the smallest homogeneous ideal
containing $I$ which is saturated.
The saturation degree $\operatorname{satdeg}(I)$ of $I$
is the least degree $s$ such that $(\operatorname{sat}(I))_t=I_t$ for all
$t\geq s$. If $I$ defines a 0-dimensional subscheme of ${\bf P}^N$,
then $\operatorname{reg}(I)$ is the maximum
of $\operatorname{satdeg}(I)$ and $\sigma(\hbox{sat}(I))$, and 
so we always have $\operatorname{reg}(I)\geq \operatorname{satdeg}(I)$
and see that $\operatorname{reg}(I)=\sigma(I)$ in the case that
$I$ is saturated (see \cite{refGGP}).

We can now recall a Postulational Containment Criterion that will be useful in this paper:

\begin{lem}[Postulational Criterion 2, \cite{refBH2}]\label{postcrit2}
Let $I\subseteq K[{\bf P}^N]$ be a homogeneous ideal 
(not necessarily saturated) defining
a 0-dimensional subscheme.
If $r\operatorname{reg}(I) \le \alpha(I^{(m)})$, then $I^{(m)} \subseteq I^r$.
\end{lem}

We also recall one of the main results of \cite{refHoHu}. 
The containment $I^{(Nt)}\subseteq I^t$ is the special case
for which $m=1$.

\begin{thm}\label{HoHuThm}
Let $I\subseteq K[\pr{N}]$ be a homogeneous ideal. Then
$I^{(t(m+N-1))}\subseteq (I^{(m)})^t$ holds for all $m,t\geq1$.
\end{thm}

Another useful fact is:

\begin{prop}\label{conj8prop}
If $I\subseteq K[\pr2]$ is a homogeneous ideal such that $I^{(j+1)}\subseteq MI^{(j)}$ 
and $I^{(2j)}=(I^{(2)})^j$ for all $j\geq1$, then $I^{(t(m+1))}\subseteq M^t(I^{(m)})^t$
holds for all $t\geq1$ and all $m\geq1$ and
$I^{(t(m+1)-1)}\subseteq M^{t-1}(I^{(m)})^t$ holds for all $t\geq1$ and all even $m\geq 2$.
If, moreover, $I^{(2j+1)}=(I^{(2)})^jI$ holds for all $j\geq0$, then
$I^{(t(m+1)-1)}\subseteq M^{t-1}(I^{(m)})^t$ holds for all $t\geq1$ and all odd $m\geq 1$.
\end{prop}

\begin{proof}
Since $I^{(j+1)}\subseteq MI^{(j)}$, we have $I^{(j+i)}\subseteq M^iI^{(j)}$ for all
$i\geq1$ and all $j\geq1$. And if $j=ab$, then $I^{(2j)}=(I^{(2)})^j=(I^{(2)})^{ab}=((I^{(2)})^a)^b=(I^{(2a)})^b$,
so whenever $m$ is even we have $I^{(tm)}=(I^{(m)})^t$.

First assume $m$ is even. Then $I^{(tm)}=(I^{(m)})^t$ and hence
$I^{(t(m+1))}\subseteq M^tI^{(tm)}=M^t(I^{(m)})^t$, and also 
$I^{(t(m+1)-1)}=I^{(tm+t-1)}\subseteq M^{t-1}I^{(tm)}=M^{t-1}(I^{(m)})^t$.

Now assume $m$ is odd. Then $I^{(t(m+1))}=(I^{(m+1)})^t$.
But $I^{(m+1)}\subseteq MI^{(m)}$, so $I^{(t(m+1))}=(I^{(m+1)})^t\subseteq (MI^{(m)})^t=M^t(I^{(m)})^t$.
Finally assume in addition that $I^{(2j+1)}=(I^{(2)})^jI$ holds for $j\geq0$ and write $m=2\mu+1$.
If $t=2\tau$ is even, then we have 
$$I^{(t(m+1)-1)}=I^{(4\mu\tau+4\tau-1)}=I^{(4\mu\tau+4(\tau-1)+2+1)}=(I^{(2)})^{2\mu\tau+2(\tau-1)+1}I
=I^{(2t\mu)}I^{(2(t-1))}I.$$
But $I^{(2(t-1))}\subseteq M^{t-1}I^{t-1}$ (applying the case already proved with $m=1$), so we have
$I^{(t(m+1)-1)}\subseteq I^{(2t\mu)}M^{t-1}I^{t-1}I=M^{t-1}I^tI^{(2t\mu)}=M^{t-1}(I^{(2\mu)}I)^t=M^{t-1}(I^{(m)})^t$.
If, on the other hand, $t=2\tau+1$ is odd, then 
$I^{(t(m+1)-1)}=(I^{(2)})^{2\mu\tau+2\tau+\mu}I=(I^{(2\mu)})^tI^{(2(t-1))}I\subseteq (I^{(2\mu)})^tM^{t-1}I^{t-1}I
=M^{t-1}(I^{(m)})^t.$
\end{proof}

The next result is a special case of Proposition 4.2.3 of \cite{refBo}.

\begin{lem}\label{EulerLemma}
Assume $K$ has characteristic 0. Let $I\subseteq K[\pr N]=K[x_0,\ldots,x_N]$ be the radical
ideal of a finite set of points $P_1,\ldots,P_n$.
Then $I^{(j+1)}\subseteq MI^{(j)}$ for each $j\geq1$.
\end{lem}

\begin{proof}
Let $F\in I^{(j+1)}$ be homogeneous.
Since $F\in (I(P_i))^{j+1}$ for each $i$, fixing $i$ and taking coordinates $x_0,\ldots,x_N$
such that $x_\ell$ vanishes at $P_i$ for $\ell>0$,
$F$ is a sum of monomials in $x_0,\ldots,x_N$ of degree at least $j+1$
in the variables $x_1,\ldots,x_N$. Thus for each $0\leq \ell\leq N$, the degree
in the variables $x_1,\ldots,x_N$ of each term of $\partial F/\partial x_\ell$ is 
at least $j$, hence $\partial F/\partial x_\ell\in (I(P_i))^{j}$. 
Since the partials with respect to one set of coordinates are linear
combinations of the partials with respect to any other linear change of coordinates,
we see for any choice of coordinates $x_0,\ldots,x_N$ on $\pr N$ that
$\partial F/\partial x_\ell\in I(P_i)^j$ for each $\ell$ and $i$, hence
$\partial F/\partial x_\ell\in I^{(j)}$ for each $\ell$.

Applying Euler's identity for a homogeneous polynomial $G$ of
positive degree (that $\deg(G)G=\sum_\ell x_\ell\partial G/\partial x_\ell$)
we see $F$ is contained in $MI^{(j)}$.
\end{proof}

This raises the following question:

\begin{ques}
Let $I\subsetneq K[\pr N]$ be any proper homogeneous ideal where $K$ has arbitrary characteristic.
Is it true that $I^{(j+1)}\subseteq MI^{(j)}$ for each $j\geq1$?
\end{ques}

The following result can be useful in some cases; it is a variation of
\cite[Proposition 2.3]{refHaHu}. 

\begin{lem}\label{genlemma} Let $I\subset K[\pr N]$ be a homogeneous ideal defining a 
zero dimensional subscheme of $\pr N$. 
\begin{itemize}
\item[(a)] If $\alpha(I^{(t(m+N-1))})\geq t\operatorname{reg}(I^{(m)})+t(N-1)$, then
$I^{(t(m+N-1))}\subseteq M^{t(N-1)}(I^{(m)})^t$.
\item[(b)] If $I^{(t(m+N-1)-(N-1))}\subseteq (I^{(m)})^t$ and
$\alpha(I^{(t(m+N-1)-(N-1))})\geq t\operatorname{reg}(I^{(m)})+(t-1)(N-1)$, then
$$I^{(t(m+N-1)-(N-1))}\subseteq M^{(t-1)(N-1)}(I^{(m)})^t.$$
\end{itemize}
\end{lem}

\begin{proof}
We know, by \cite[Theorem 1.1]{refGGP}, that
$\operatorname{reg}((I^{(m)})^t)\leq t\operatorname{reg}(I^{(m)})$.
In particular, $(I^{(m)})^t$ is generated in degree at most $t\operatorname{reg}(I^{(m)})$.
Thus for any degree $s\geq t\operatorname{reg}(I^{(m)})$ we have
$M_1((I^{(m)})^t)_s=((I^{(m)})^t)_{s+1}$ and hence
$(M^i)_i((I^{(m)})^t)_s=((I^{(m)})^t)_{s+i}$.
But $(M^i)_i((I^{(m)})^t)_s\subseteq (M^i(I^{(m)})^t)_{s+i}\subseteq ((I^{(m)})^t)_{s+i}$ 
so $(M^i(I^{(m)})^t)_{s+i}=((I^{(m)})^t)_{s+i}$
if $s\geq t\operatorname{reg}(I^{(m)})$. 

(a) By Theorem \ref{HoHuThm} we have $I^{(t(m+N-1))}\subseteq (I^{(m)})^t$.
Thus we have $(I^{(t(m+N-1))})_s\subseteq ((I^{(m)})^t)_s=(M^{t(N-1)}(I^{(m)})^t)_s$
for $s\geq t\operatorname{reg}(I^{(m)})+t(N-1)$. But
$(I^{(t(m+N-1))})_s=0$ for $s<\alpha(I^{(t(m+N-1))})$, so if
$\alpha(I^{(t(m+N-1))})\geq t\operatorname{reg}(I^{(m)})+t(N-1)$, then we have
$(I^{(t(m+N-1))})_s\subseteq ((I^{(m)})^t)_s$ for all $s\geq0$,
which implies the result.

(b) By assumption we have $I^{(t(m+N-1)-(N-1))}\subseteq (I^{(m)})^t$.
Now mimic the proof of (a). We have $(I^{(t(m+N-1)-(N-1))})_s\subseteq ((I^{(m)})^t)_s=(M^{(t-1)(N-1)}(I^{(m)})^t)_s$
for $s\geq t\operatorname{reg}(I^{(m)})+(t-1)(N-1)$. But
$(I^{(t(m+N-1)-(N-1))})_s=0$ for $s<\alpha(I^{(t(m+N-1)-(N-1))})$, so if
$\alpha(I^{(t(m+N-1)-(N-1))})\geq t\operatorname{reg}(I^{(m)})+(t-1)(N-1)$, then we have
$(I^{(t(m+N-1)-(N-1))})_s\subseteq ((I^{(m)})^t)_s$ for all $s\geq0$,
which implies the result.
\end{proof}


\section{The conjectures}\label{theconjectures}

For the reader's convenience, we list here the conjectures we will be considering.

\begin{conj}[{\cite[Conjecture 2.1]{refHaHu}}]\label{fatptconj}
Let $I=\cap_{i=1}^n I(P_i)^{m_i}\subset K[\pr N]$ be any fat points ideal.
Then $I^{(rN)}\subseteq M^{r(N-1)}I^r$
holds for all $r>0$. 
\end{conj}

\begin{conj}[{\cite[Conjecture 8.20]{refB. et al}}]\label{Essenconj}
Let $I\subseteq K[\pr N]$ be a homogeneous ideal.
Then $I^{(rN-(N-1))}\subseteq I^r$ holds for all $r$. 
\end{conj}

\begin{conj}[{\cite[Conjecture 4.1.4]{refHaHu}}]\label{p2conj}
Let $I\subseteq K[\pr 2]$ be the radical ideal of a finite set of $n$ points $P_i\in\pr 2$.
Then $I^{(m)}\subseteq I^r$ holds whenever $m/r\ge 2\alpha(I)/(\alpha(I)+1)$. 
\end{conj}

\begin{conj}[{\cite[Conjecture 4.1.5]{refHaHu}}]\label{EvoEssenconj}
Let $I\subseteq K[\pr N]$ be the radical ideal of a finite set of $n$ points $P_i\in\pr N$.
Then $I^{(rN-(N-1))}\subseteq M^{(r-1)(N-1)}I^r$ holds for all $r\geq1$. 
\end{conj}

\begin{conj}[{\cite[Conjecture 4.1.8]{refHaHu}}]\label{EvoEssenconj2}
Let $I\subseteq K[\pr N]$ be the radical ideal of a finite set of $n$ points $P_i\in\pr N$.
Then 
$$\alpha(I^{(rN-(N-1))})\ge r\alpha(I)+(r-1)(N-1)$$ 
for every $r>0$.
\end{conj}

\begin{conj}[{\cite[Question 4.2.1]{refHaHu}}]\label{refinedChud}
Let $I\subseteq K[\pr N]$ be the radical ideal of a finite set of $n$ points $P_i\in\pr N$.
Then  
$$\frac{\alpha(I^{(m)})+N-1}{m+N-1}\leq \frac{\alpha(I^{(r)})}{r}$$
for all $r>0$.
\end{conj}

\begin{conj}[{\cite[Question 4.2.2]{refHaHu}}]\label{refinedEV}
Let $I\subseteq K[\pr N]$ be the radical ideal of a finite set of $n$ points $P_i\in\pr N$ for $N\geq 2$. Then  
$I^{(t(m+N-1))}\subseteq M^t(I^{(m)})^t$ for all positive integers $m$ and $t$.
\end{conj}

\begin{conj}[{\cite[Question 4.2.3]{refHaHu}}]\label{refinedEV2}
Let $I\subseteq K[\pr N]$ be the radical ideal of a finite set of $n$ points $P_i\in\pr N$. Then 
$I^{(t(m+N-1))}\subseteq M^{t(N-1)}(I^{(m)})^t$ for all positive integers $m$ and $t$.
\end{conj}

Note that if we take $m=1$ in Conjecture \ref{refinedEV2} we recover Conjecture \ref{fatptconj}.
This suggests the following conjecture, which in the same way implies Conjectures \ref{Essenconj} 
and \ref{EvoEssenconj}. It also completes a pair of analogies: 
Conjecture \ref{fatptconj} is to 
Conjecture \ref{EvoEssenconj} as Conjecture \ref{refinedEV2} is to the second part of the following conjecture,
and Conjecture \ref{Essenconj} is to 
Conjecture \ref{EvoEssenconj} as the first part of the following conjecture is to 
the second part.

\begin{conj}\label{ourconj}
Let $I\subseteq K[\pr N]$ be the radical ideal of a finite set of $n$ points $P_i\in\pr N$. Then 
$I^{(t(m+N-1)-N+1)}\subseteq (I^{(m)})^t$ and $I^{(t(m+N-1)-N+1)}\subseteq M^{(t-1)(N-1)}(I^{(m)})^t$
hold for all positive integers $m$ and $t$.
\end{conj}

As we just saw, some conjectures are stronger than others. In fact the following implications hold:

\begin{prop}\ 
\begin{itemize}
\item[(1)] Conjecture \ref{EvoEssenconj} implies Conjecture \ref{EvoEssenconj2}.
\item[(2)] Conjecture \ref{refinedEV2} implies Conjectures \ref{refinedChud} and \ref{refinedEV},
and, when $I$ is radical, Conjecture \ref{fatptconj}.
\item[(3)] Conjecture \ref{ourconj} implies Conjectures \ref{EvoEssenconj} and \ref{EvoEssenconj2},
and, when $I$ is the radical ideal of a finite set of points, Conjecture \ref{Essenconj}.
\end{itemize}
\end{prop}

\begin{proof}
(1) Suppose Conjecture \ref{EvoEssenconj}  holds. From $I^{(rN-(N-1))}\subseteq M^{(r-1)(N-1)}I^r$ one has
\[
\alpha(I^{(rN-(N-1))})\geq \alpha(M^{(r-1)(N-1)}I^r)=\alpha(M^{(r-1)(N-1)})+\alpha(I^r)=(r-1)(N-1)+r\alpha(I)
\]
and hence Conjecture \ref{EvoEssenconj2} holds.

(2) Suppose Conjecture \ref{refinedEV2} holds. The implication of Conjecture \ref{refinedEV}  is obvious. 
By Conjecture  \ref{refinedEV2}, taking $t=r$ and $m=1$, 
one has $I^{(rN)}\subseteq M^{r(N-1)}(I^{(m)})^r=M^{r(N-1)}I^r$. 
Hence Conjecture  \ref{fatptconj} holds. 
Again from $I^{(t(m+N-1))}\subseteq M^{t(N-1)}(I^{(m)})^t$ of  
Conjecture  \ref{refinedEV2}, and from $(I^{(t)})^{m+N-1}\subseteq I^{(t(m+N-1))}$ one has
\[
t(\alpha(I^{(m)})+N-1)=\alpha(M^{t(N-1)}(I^{(m)})^t)\leq \alpha((I^{(t)})^{m+N-1})=(m+N-1)\alpha(I^{(t)})
\]
from which Conjecture \ref{refinedChud} follows.

(3) Conjecture \ref{ourconj} implies Conjectures 
\ref{Essenconj} and \ref{EvoEssenconj} by taking $m=1$; then note that
Conjecture \ref{EvoEssenconj} implies Conjecture \ref{EvoEssenconj2} by (1).
\end{proof}

\begin{rem}\label{cntrex}
While this paper was under review a counterexample to Conjecture \ref{Essenconj}
over the complex numbers in the case that $r=N=2$ was posted to the arXiv \cite{refDST}.
As a result, some adjustment to the statements of Conjectures \ref{Essenconj}, 
\ref{EvoEssenconj} and \ref{ourconj} will be needed.
The counterexample involves a specific configuration of 12 points in $\pr2_{\bf C}$. 
These 12 points are arranged with respect to 9 lines such that each line 
contains 4 of the points and each point lies on 3 of the lines. 
A similar counterexample arises by taking any 12 of the 13 points 
of the projective plane $\pr2_K$ over the field $K={\bf Z}/3{\bf Z}$ of three elements. 
(Note that $\pr2_{K}$ has 13 points and 13 lines. Removing one point leaves 12 points 
and removing the 4 lines through that point leaves 9 lines.
These 12 points and 9 lines also are such that each line 
contains 4 of the points and each point lies on 3 of the lines.) It turns out, 
both for the counterexample given in \cite{refDST}
and for 12 points in $\pr2_{K}$ with $K={\bf Z}/3{\bf Z}$, 
that the form $F$ obtained as the product of the 9 linear forms
defining the lines is in $I^{(3)}$ but not in $I^2$, where $I$ is the ideal of the 12 points.
Thus $I^{(3)}\not\subseteq I^2$. This also means that $I^{(3)}\not\subseteq MI^2$, and hence
gives rise to a counterexample to Conjecture \ref{EvoEssenconj} when $N=r=2$ 
and to Conjecture \ref{ourconj} when $N=r=2$ and $m=1$. 

However, by \cite[Examples 8.4.4 and 8.4.5]{refB. et al}, Conjecture \ref{Essenconj}
holds when $I$ is an ideal of points over a ground field of characteristic 2,
and when $I$ is a monomial ideal in any characteristic. 
This indicates that the conjecture in some cases depends
on the characteristic, but also that it holds in many cases. 
In this paper, we verify other cases for which the conjectures hold.
It seems that any failures must be fairly special, like the 12 points discussed above.
These observations raise some interesting new lines of investigation. 
Can the failures, or the characteristics in which they occur, be classified?
Alternatively, do Conjectures \ref{Essenconj}, \ref{EvoEssenconj} 
and \ref{ourconj} hold at least when $r>N$? 
Computer calculations suggest this may be the case for the 12 points in 
$\pr2_K$ over $K={\bf Z}/3{\bf Z}$.
\end{rem}

\section{Assuming symbolic equals ordinary}\label{symisord}

Suppose $I\subsetneq K[\pr{N}]$ is a homogeneous ideal such that $I^{(r)}=I^r$ for all $r$, such as is 
the case if $I$ is a complete intersection (see \cite[Lemma 5, Appendix 6]{refZS}). Note that $I\subseteq M$.

Conjectures \ref{fatptconj}, \ref{Essenconj}, \ref{EvoEssenconj}, 
\ref{refinedEV}, \ref{refinedEV2} and \ref{ourconj} all hold in this situation, by
very similar proofs. We demonstrate the method by proving Conjecture \ref{ourconj}.
The second part of Conjecture \ref{ourconj},
$I^{(t(m+N-1)-N+1)}\subseteq M^{(t-1)(N-1)}(I^{(m)})^t$,
holds (and hence so does the first part
$I^{(t(m+N-1)-N+1)}\subseteq (I^{(m)})^t$), since 
$$I^{(t(m+N-1)-N+1)}=I^{t(m+N-1)-N+1}=I^{(t-1)(N-1)}I^{mt}\subseteq 
M^{(t-1)(N-1)}I^{mt}=M^{(t-1)(N-1)}(I^{(m)})^t.$$

Now consider Conjecture \ref{p2conj}, that $m/r\ge 2\alpha(I)/(\alpha(I)+1)$
implies $I^{(m)}\subseteq I^r$ for $N=2$. This holds 
since $I^{(m)}=I^m$ and since $2\alpha(I)/(\alpha(I)+1)\geq 1$, so
$m/r\ge 2\alpha(I)/(\alpha(I)+1)$ implies $m\geq r$, hence $I^m\subseteq I^r$.

Conjectures \ref{EvoEssenconj2} and \ref{refinedChud} have very similar proofs.
We leave \ref{EvoEssenconj2} to the reader. Consider \ref{refinedChud}, that
$\frac{\alpha(I^{(m)})+N-1}{m+N-1}\leq \frac{\alpha(I^{(r)})}{r}$.
This holds since
$\alpha(I)\geq 1$ but
$$\frac{m\alpha(I)+N-1}{m+N-1}=\frac{\alpha(I^{(m)})+N-1}{m+N-1}\leq
\frac{\alpha(I^{(r)})}{r}=\frac{\alpha(I^r)}{r}=\frac{r\alpha(I)}{r}= \alpha(I)$$ 
is equivalent to $m\alpha(I)+N-1\leq \alpha(I)(m+N-1)$ and hence to $1\leq \alpha(I)$.

In particular, if $N=2$, all of the conjectures hold if $I$ is the radical ideal 
of a single point, or the radical ideal of $n$ points on a 
smooth plane conic when $n$ is even,
since then the points comprise a complete intersection.
In the next section we consider the case of the radical ideal of an odd number of points 
on a smooth plane conic. (One can also ask about the case of ideals
of points on reducible conics; i.e., of points on a pair of lines in the plane. 
See \cite{refDJ} for some results in this direction.)

\section{Odd numbers of points on a smooth plane conic}\label{oddsmoothconic}

The case of $n=3$ points on a smooth plane conic $C$ (so $N=2$) is somewhat special and mostly
known, so we treat that
case with some initial remarks. Conjecture \ref{fatptconj} holds by \cite[Corollary 3.9]{refHaHu}
since $n=3$ points on a smooth conic comprise a star configuration.
Conjectures \ref{Essenconj} and \ref{EvoEssenconj2}
follow from Conjecture \ref{EvoEssenconj}, which holds
for the $n=3$ case by \cite[Corollary 4.1.7]{refHaHu}.
Conjecture \ref{p2conj} holds for $n=3$ since $\alpha(I)=2$ so
$m/r\ge 2\alpha(I)/(\alpha(I)+1)$ is equivalent to
$3m/2\geq2r$, but $\operatorname{reg}(I)=2$ and, by \cite[Lemma 2.4.1]{refBH2},
$3m/2\leq \alpha(I^{(m)})$, so assuming $3m/2\geq2r$ we have
$r\operatorname{reg}(I)=2r\leq 3m/2\leq \alpha(I^{(m)})$,
hence $I^{(m)}\subseteq I^r$ by Lemma \ref{postcrit2}.
Conjectures \ref{refinedChud} and \ref{refinedEV} follow from
Conjecture \ref{refinedEV2}, which holds for $n=3$ in characteristic 0
by Section \ref{stars}.

Before continuing, we introduce another useful numerical character:  
for any homogeneous ideal $0 \not = I \subseteq R = K[\pr 2]$, let $\beta(I)$ be the 
smallest integer $t$ such that $I_t$ contains a regular sequence of length two.

\begin{lem}\label{evenlem}
Let $I$ be the radical ideal of $n\geq5$ points on a smooth conic $C$ in $\pr2$, where $n$ is odd.
Then $I^{(mr)}=(I^{(m)})^r$ for any even $m$.
\end{lem}

\begin{proof}
We first note that $I^{(2r)}=(I^{(2)})^r$. This is because $\alpha(I^{(2r)})=4r$
and $\beta(I^{(2r)})=rn$ (because it is known \cite{refHa}
that the only fixed component
for the linear system of curves of degree $d$ through $n\geq 5$ points of multiplicity $r$
is the conic $C$ through all $n$ points,
and these occur only if forced by B\'ezout).
Now we have $\alpha(I^{(2r)})\beta(I^{(2r)})=(2r)^2n$, hence
$I^{(2r)}=(I^{(2)})^r$ by Proposition 3.5 of \cite{refHaHu}.
Thus for $m=2s$ we have $I^{(mr)}=I^{(2sr)}=(I^{(2)})^{sr}=((I^{(2)})^s)^r=(I^{(m)})^r$.
\end{proof}

We now deal with the case of odd symbolic powers:

\begin{lem}\label{oddlem}
If $I$ is the radical ideal of $n\geq5$ points on a smooth conic $C$ in $\pr2$ with $n$ odd,
then $I^{(2r+1)}=(I^{(2)})^rI$.
\end{lem}

\begin{proof}
It is enough to show that $I^{(2r+1)}=(I^{(2r)})I$, since $I^{(2r)}=(I^{(2)})^r$.
Since $(I^{(2r)})I\subseteq I^{(2r+1)}$, it is enough to show 
$(I^{(2r+1)})_t\subseteq((I^{(2r)})I)_t$ for every $t\geq\alpha(I^{(2r+1)})=2(2r+1)$.
Let $m=2r+1$ and let $F$ be the form defining $C$. If $4r+2\leq t\leq (mn-1)/2$, then
$2t-mn<0$, so by B\'ezout we know that $C$ is a fixed component of
$(I^{(2r+1)})_t$, hence $(I^{(2r+1)})_t=F\cdot (I^{(2r)})_{t-2}\subseteq (I^{(2r)})_{t-2}I_2
\subseteq ((I^{(2r)})I)_t$. So say $t\geq (mn+1)/2$. Note that $I_{(n+1)/2}$
is fixed component free since $2(n+1)/2-n\geq0$, and $(I^{(m-1)})_j$
is fixed component free for $j\geq n(m-1)/2$, since 
$2j-(m-1)n\geq 2n(m-1)/2-(m-1)n\geq0$, so by 
\cite[Proposition 2.4]{refBH} we have $(I^{(m-1)})_jI_{(n+1)/2}=
(I^{(m)})_{j+((n+1)/2)}$. 
But let $j=t-(n+1)/2$. Then $t\geq (mn+1)/2$ implies
$j\geq n(m-1)/2$, so we have $(I^{(m)})_t=(I^{(m-1)})_jI_{(n+1)/2}
\subseteq (I^{(m-1)}I)_t$.
\end{proof}

We now show that Conjectures \ref{fatptconj} through \ref{refinedChud} hold
in the case that $N=2$ for $n\geq5$ points on a smooth plane conic if $n$ is odd.

\ \newline\noindent Conjecture \ref{fatptconj}, $I^{(rN)}\subseteq M^{r(N-1)}I^r$, holds:
Suppose we show $I^{(2)}\subseteq MI$. Then we would have 
$I^{(2r)}=(I^{(2)})^r\subseteq (MI)^r=M^rI^r$, as required.
Thus it's enough now to show that $(I^{(2)})_t\subseteq M_iI_{t-i}$ 
for some $i$ for each $t\geq \alpha(I^{(2)})$.
But for $\alpha(I^{(2)})\leq t<\beta(I^{(2)})=n$, 
we know $C$ is a fixed component of
$(I^{(2)})_t$. Since $F\in M_2$ for the form $F$ defining $C$, we have
$(I^{(2)})_t=F(I_{t-2})\subseteq M_2I_{t-2}$, as needed.
For $t\geq n$, we have $M_1I_{t-1}=I_t$ by \cite[Lemma 2.2]{refBH}
since $t\geq n=\beta(I^{(2)})>\beta(I)$,
so $(I^{(2)})_t\subseteq I_t=M_1I_{t-1}$.

\ \newline\noindent Conjecture \ref{Essenconj}, $I^{(rN-(N-1))}\subseteq I^r$, holds:
Here we want to verify that $I^{(2r-1)}\subseteq I^r$, which is 
$I^{(2m+1)}\subseteq I^{m+1}$ if we take $r=m+1$.
But $I^{(2m+1)}=(I^{(2)})^mI\subseteq (MI)^{m}I=M^mI^{m+1}\subseteq I^{m+1}$,
as required.

\ \newline\noindent Conjecture \ref{p2conj}, that $m/r\ge 2\alpha(I)/(\alpha(I)+1)$
implies $I^{(m)}\subseteq I^r$, holds: To see this it is helpful to recall the resurgence,
$\rho(J)$ of a homogeneous ideal $(0)\neq J\subsetneq k[\pr N]$, defined to 
be the supremum of the ratios $m/r$ such that $I^{(m)}\not\subseteq I^r$. 
In the present situation, it is enough to have $\rho(I)<2\alpha(I)/(\alpha(I)+1)$, 
but by \cite[Theorem 3.4]{refBH}
we have $\rho(I)=(n+1)/n$. Since $\alpha(I)=2$ and $n\geq 5$ we have
$\rho(I)=(n+1)/n<4/3=2\alpha(I)/(\alpha(I)+1)$.

\ \newline\noindent Conjecture \ref{EvoEssenconj2}
follows from Conjecture \ref{EvoEssenconj}. For Conjecture \ref{EvoEssenconj}
we want to verify that $I^{(2r-1)}\subseteq M^{r-1}I^r$, which is 
just $I^{(2m+1)}\subseteq M^mI^{m+1}$ if we take $r=m+1$.
But $I^{(2m+1)}=(I^{(2)})^mI\subseteq (MI)^{m}I=M^mI^{m+1}$,
as required.

\ \newline\noindent Conjecture \ref{refinedChud}, $\frac{\alpha(I^{(m)})+N-1}{m+N-1}\leq \frac{\alpha(I^{(r)})}{r}$, holds:
Here we want to verify that $\frac{\alpha(I^{(m)})+1}{m+1}\leq \frac{\alpha(I^{(r)})}{r}$ 
for all $r\geq1$, but 
$$\frac{\alpha(I^{(m)})+1}{m+1}=\frac{2m+1}{m+1}\leq 2=\frac{2r}{r}=\frac{\alpha(I^{(r)})}{r}$$
is clear.

In case $N=2$, Conjecture \ref{refinedEV} and Conjecture \ref{refinedEV2} both
assert that $I^{(t(m+1))}\subseteq M^{t}(I^{(m)})^t$, while 
Conjecture \ref{ourconj} asserts $I^{(t(m+1)-1)}\subseteq M^{t-1}(I^{(m)})^t\subseteq (I^{(m)})^t$.
We note that the next result holds also for $n=1$ and $n=3$, but when $n=1$ we have a single point
and that case was dealt with in Section \ref{symisord}, while for $n=3$ the points comprise
a star configuration, which is dealt with in Section \ref{stars}.

\begin{cor}\label{char0cor} Let $N=2$ and let $n\geq 5$ be odd. Then 
Conjectures \ref{refinedEV}, \ref{refinedEV2} and \ref{ourconj} hold for the radical ideal of any
$n$ points on a smooth plane conic if $\operatorname{char}(K)=0$.
\end{cor}

\begin{proof}
Let $I$ be the radical ideal of the points. 
For $n\geq5$ odd we have $I^{(2j)}=(I^{(2)})^j$
by Lemma \ref{evenlem}, and we have $I^{(2r+1)}=(I^{(2)})^rI$
by Lemma \ref{oddlem}. The result now follows in characteristic 0
from Proposition \ref{conj8prop} by applying Lemma \ref{EulerLemma}.
\end{proof}

\section{Star configurations of points in projective space}\label{stars}

Let $I$ be the radical ideal of the configuration of $\binom{s}{N}$ points given by the 
pair-wise intersection of $s\geq N$ hyperplanes in $\pr N$,
assuming no $N+1$ of the hyperplanes meet at a point.

\ \newline\noindent Conjecture \ref{fatptconj}, that $I^{(rN)}\subseteq M^{r(N-1)}I^r$, holds by \cite[Corollary 3.9]{refHaHu}.

\ \newline\noindent Conjecture \ref{Essenconj}, that $I^{(rN-(N-1))}\subseteq I^r$, holds by \cite[Example 8.4.8]{refB. et al}.

\ \newline\noindent Conjecture \ref{p2conj}, that $I^{(m)}\subseteq I^r$ if $m/r\ge 2\alpha(I)/(\alpha(I)+1)$ and $N=2$,
holds (see the discussion after Conjecture 4.1.4 of \cite{refHaHu}).

\ \newline\noindent Conjecture \ref{EvoEssenconj}, that $I^{(rN-(N-1))}\subseteq M^{(r-1)(N-1)}I^r$, holds by
\cite[Corollary 4.1.7]{refHaHu}.

\ \newline\noindent Conjecture \ref{EvoEssenconj2}, that $\alpha(I^{(rN-(N-1))})\ge r\alpha(I)+(r-1)(N-1)$, 
follows from Conjecture \ref{EvoEssenconj}.

\ \newline\noindent Conjecture \ref{refinedChud}, that $\frac{\alpha(I^{(m)})+N-1}{m+N-1}\leq \frac{\alpha(I^{(r)})}{r}$, and
Conjecture \ref{refinedEV}, that $I^{(t(m+N-1))}\subseteq M^t(I^{(m)})^t$ for $N\geq 2$, follow from
Conjecture \ref{refinedEV2}, that $I^{(t(m+N-1))}\subseteq M^{t(N-1)}(I^{(m)})^t$. We will verify
Conjecture \ref{refinedEV2} in case $N=2$ when $K$ has characteristic 0. 
For $N=2$, we have $I^{(2m)}=(I^{(2)})^m$ for all $m\geq1$
by \cite[Corollary 3.9]{refHaHu}, and we have $I^{(m+1)}\subseteq MI^{(m)}$ by
Lemma \ref{EulerLemma}. Thus Conjecture \ref{refinedEV2} holds when $N=2$ and
$\operatorname{char}(K)=0$ by Proposition \ref{conj8prop}.
This same proof shows that
Conjecture \ref{ourconj}, that $I^{(t(m+N-1)-1)}\subseteq M^{(t-1)(N-1)}(I^{(m)})^t$, 
holds in case $N=2$ when $K$ has characteristic 0 and $m$ is even, since 
for a star configuration in $\pr2$ we have $I^{(2j)}=(I^{(2)})^j$
by \cite[Corollary 3.9]{refHaHu}. It also holds for $m$ odd by the same proof, except
to apply Proposition \ref{conj8prop} we must check that 
$I^{(2j+1)}=(I^{(2)})^jI$, which we now do. We always have $(I^{(2)})^jI\subseteq I^{(2j+1)}$
so we must check the reverse containment. The case $j=0$ is clear, so assume $j>0$.
By \cite[Lemma 2.4.2]{refBH2},
$s-1=\alpha(I)=\operatorname{reg}(I)$ and $\operatorname{reg}(I^m)\leq m\operatorname{reg}(I)$
by \cite[Theorem 1.1]{refGGP}.
Thus we have equality of the homogeneous components of the ideals $(I^{(m)})_t=(I^m)_t$
for every degree $t\geq m(s-1)$. In particular, $(I^{(2j+1)})_t=(I^{2j+1})_t\subseteq ((I^{(2)})^jI)_t$.
For $t<m(s-1)$, by B\'ezout's Theorem,
each of the $s$ lines is in the zero locus of each element of $(I^{(2j+1)})_t$; i.e.,
$(I^{(2j+1)})_t=F(I^{(2(j-1)+1)})_{t-s}$, where the form $F$
defines the $s$ lines. By induction we have
$F(I^{(2(j-1)+1)})_{t-s}\subseteq F(I^{(2(j-1))}I)_{t-s}\subseteq (I^{(2j)}I)_t=((I^{(2)})^jI)_t$.
Thus $(I^{(2j+1)})_t\subseteq ((I^{(2)})^jI)_t$ holds for all $t$, hence
$I^{(2j+1)}\subseteq(I^{(2)})^jI$.

\section{Points in a hyperplane}\label{HyperplaneSection} 

Suppose distinct points $P_1,\ldots,P_n\in \pr {N+1}$ lie in a hyperplane $H$ of $\pr{N+1}$.
We can regard $H$ as $\pr N$. There is now some ambiguity when using the notation $m_1P_1+\cdots+m_nP_n$,
since $m_1P_1+\cdots+m_nP_n$ could denote a fat point subscheme of $\pr N$ or of
$\pr{N+1}$, and these are not the same. For example, consider the point $P=(0,0,0,1)\in\pr3$.
Taking coordinates $K[\pr3]=K[x_0,\ldots,x_3]$, $P\in H$ where $H$ is defined as $x_0=0$.
We can identify $H$ with $\pr2$, where $K[\pr2]=K[x_1,\ldots,x_3]$.
If we use $mP$ to denote the fat point subscheme of $\pr3$, we have $I(mP)=(x_0,x_1,x_2)^m$, 
but, as a fat point subscheme of $\pr2$ which becomes a subscheme of $\pr3$ by
the inclusion $\pr2\subset\pr3$, $mP$ has ideal $I(mP)=(x_1,x_2)^m+(x_0)$.

To resolve the ambiguity, given $Z=m_1P_1+\cdots+m_nP_n$ for points $P_i\in H$,
we will use the notation $Z_{\pr N}$ and $Z_{\pr{N+1}}$ to indicate whether we are regarding
$Z$ as the fat point subscheme of $\pr N$ or of $\pr{N+1}$, respectively.

For explicitness (but without loss of generality) we assume 
$K[\pr{N+1}]=K[x_0,\ldots,x_{N+1}]$ and $H$ is defined by $x_0=0$, so
$K[\pr N]=K[H]=K[x_1,\ldots,x_{N+1}]$. 

Let $Z=P_1+\cdots+P_n$, let $I=I(Z_{\pr N})$, let $\widetilde{I}=IK[\pr{N+1}]$
be the extension and let $\widehat{I}=I(Z_{\pr {N+1}})$. Also,
let $M=(x_1,\ldots,x_{N+1})\subset K[\pr N]$,
let $\widetilde{M}=MK[\pr{N+1}]$ and let $\widehat{M}=(x_0,\ldots,x_{N+1})\subset K[\pr{N+1}]$.

\begin{prop}\label{hypprop} Let $I, \widetilde{I},\widehat{I},M,\widetilde{M}$ and $\widehat{M}$ 
be as above. Moreover, for parts (a) and (c), assume $K$ has characteristic 0.
\begin{itemize}
\item[(a)] If Conjecture \ref{fatptconj} holds for $I$, then it holds for $\widehat{I}$; i.e., 
if $I^{(rN)}\subseteq M^{r(N-1)}I^r$ holds, then so does 
$\widehat{I}^{(r(N+1))}\subseteq \widehat{M}^{rN}\widehat{I}^r$.
\item[(b)] If Conjecture \ref{Essenconj} holds for $I$, then it holds for $\widehat{I}$; in fact,
if $I^{(rN-(N-1))}\subseteq I^r$ holds for all $r\geq 1$, then so does 
$\widehat{I}^{(r(N+1)-N)}\subseteq \widehat{I}^{(rN-(N-1))}\subseteq \widehat{I}^r$.
\item[(c)] If Conjecture \ref{EvoEssenconj} holds for $I$, then it holds for $\widehat{I}$; i.e.,
if $I^{(rN-(N-1))}\subseteq M^{(r-1)(N-1)}I^r$ holds, then so does 
$\widehat{I}^{(r(N+1)-N)}\subseteq \widehat{M}^{(r-1)N}\widehat{I}^r$.
\end{itemize} 
\end{prop}

\begin{proof}(a) By \cite{refFHL}, $\widehat{I}^{(m)}=(x_0^m)+x_0^{m-1}\widetilde{I}+\cdots 
+x_0\widetilde{I}^{(m-1)}+ \widetilde{I}^{(m)}$ for any $m\geq 1$. Thus
$\widehat{I}=(x_0)+\widetilde{I}$, so 
$\widehat{M}^{rN}\widehat{I}^r$ is equal to $((x_0)+\widetilde{M})^{rN}((x_0)+\widetilde{I})^r$,
which expands to 
$$((x_0)^{rN}+x_0^{rN-1}\widetilde{M}+\cdots+x_0\widetilde{M}^{rN-1}+
\widetilde{M}^{rN})((x_0)^r+x_0^{r-1}\widetilde{I}+\cdots+x_0\widetilde{I}^{r-1}
+\widetilde{I}^r).$$
Multiplying this out and combining terms with equal powers of $x_0$ gives
\[
\begin{split}
&(x_0)^{(N+1)r}+\\
&(x_0)^{(N+1)r-1}(\widetilde{M}+\widetilde{I})+\\
&(x_0)^{(N+1)r-2}(\widetilde{M}^2+\widetilde{M}\widetilde{I}+\widetilde{I}^2)+\\
&\cdots +\\
&(x_0)^{(N+1)r-r}(\widetilde{M}^r+\widetilde{M}^{r-1}\widetilde{I}+\cdots+\widetilde{I}^r)+\\
&(x_0)^{rN-1}(\widetilde{M}^{r+1}+\widetilde{M}^r\widetilde{I}+\cdots+\widetilde{M}\widetilde{I}^r)+\\
&(x_0)^{rN-2}(\widetilde{M}^{r+2}+\widetilde{M}^{r+1}\widetilde{I}+\cdots+\widetilde{M}^2\widetilde{I}^r)+\\
&\cdots+\\
&(x_0)^{rN-(rN-r)}(\widetilde{M}^{r+rN-r}+\widetilde{M}^{r+rN-r-1}\widetilde{I}+\cdots+\widetilde{M}^{r+rN-2r}\widetilde{I}^r)+\\
&(x_0)^{r-1}(\widetilde{M}^{rN}\widetilde{I}+\widetilde{M}^{rN-1}\widetilde{I}^2+\cdots+\widetilde{M}^{rN-(r-1)}\widetilde{I}^r)+\\
&\cdots+\\
&(x_0)(\widetilde{M}^{rN}\widetilde{I}^{r-1}+\widetilde{M}^{rN-1}\widetilde{I}^r)+\\
&\widetilde{M}^{rN}\widetilde{I}^r.
\end{split}
\]
Since $\widetilde{I}\subseteq \widetilde{M}$, the first term in each row of the displayed sum of ideals above
contains the other terms in that row; e.g., in row 3 we have 
$$\widetilde{I}^2\subseteq \widetilde{M}\widetilde{I}\subseteq\widetilde{M}^2.$$
Thus we get
$$\widehat{M}^{rN}\widehat{I}^r=(x_0)^{(N+1)r}+(x_0)^{(N+1)r-1}\widetilde{M}+\cdots+
(x_0)^r\widetilde{M}^{rN}+(x_0)^{r-1}\widetilde{M}^{rN}\widetilde{I}
+\cdots+x_0\widetilde{M}^{rN}\widetilde{I}^{r-1}+\widetilde{M}^{rN}\widetilde{I}^r.$$
Since
$$\widehat{I}^{(r(N+1))}=(x_0^{(N+1)r})+x_0^{(N+1)r-1}\widetilde{I}+\cdots 
+x_0\widetilde{I}^{((N+1)r-1)}+ \widetilde{I}^{((N+1)r)},$$
to show $\widehat{I}^{(r(N+1))}\subseteq\widehat{M}^{rN}\widehat{I}^r$ it suffices to show
that $x_0^{(N+1)r-j}\widetilde{I}^{(j)}\subseteq x_0^{(N+1)r-j}\widetilde{M}^j$ for $j=0,\ldots,Nr$,
and that $x_0^{(N+1)r-j}\widetilde{I}^{(j)}\subseteq x_0^{(N+1)r-j}\widetilde{M}^{Nr}\widetilde{I}^{j-Nr}$ 
for $j=Nr+1,\ldots,(N+1)r$.

However, $\widetilde{I}\subseteq \widetilde{M}$, 
so $\widetilde{I}^t\subseteq \widetilde{M}^t$ for all $t>0$.
Since adding a variable to a polynomial ring gives a flat extension,
primary decompositions of ideals in $K[\pr N]$ extend 
to primary decompositions in $K[\pr{N+1}]$ (see \cite{refM}, Theorem 13, or Exercise 7, \cite{refAM}).
Thus saturating with respect to $\widetilde{M}$ gives $\widetilde{I}^{(t)}\subseteq \widetilde{M}^{(t)}$,
but $\widetilde{M}^t$ is $\widetilde{M}$-primary so saturation has no effect;
i.e., $\widetilde{M}^{(t)}=\widetilde{M}^t$. Thus $\widetilde{I}^{(t)}\subseteq \widetilde{M}^t$, 
which shows $x_0^{(N+1)r-j}\widetilde{I}^{(j)}\subseteq x_0^{(N+1)r-j}\widetilde{M}^j$.  
Now we want to show
$$x_0^{(N+1)r-j}\widetilde{I}^{(j)}\subseteq x_0^{(N+1)r-j}\widetilde{M}^{Nr}\widetilde{I}^{j-Nr}.$$
Given that $j>Nr$, by $j-Nr$ applications of Lemma \ref{EulerLemma}, we have 
$I^{(j)}\subseteq M^{j-Nr}I^{(Nr)}$ and hence
$$x_0^{(N+1)r-j}\widetilde{I}^{(j)}\subseteq x_0^{(N+1)r-j}\widetilde{M}^{j-Nr}\widetilde{I}^{(Nr)}.$$
(Here we use the obvious fact that an extension of a product is the product of the extensions,
i.e., $(M^{j-Nr}I^{(Nr)})\widetilde{\ }=\widetilde{M}^{j-Nr}\widetilde{I^{(Nr)}}$,
and the less obvious fact that in this situation the extension of the symbolic 
power is the symbolic power of the extension; i.e., $\widetilde{I^{(Nr)}}=\widetilde{I}^{(Nr)}$. 
The latter is true because of the fact quoted above about the preservation of
primary decompositions for the extension $K[\pr N]\subset K[\pr{N+1}]$.)
Thus $\widetilde{I}^{(Nr)}\subseteq \widetilde{M}^{r(N-1)}\widetilde{I}^r$ since
by assumption we have $I^{(Nr)}\subseteq M^{r(N-1)}I^r$, hence we obtain
$$x_0^{(N+1)r-j}\widetilde{I}^{(j)}\subseteq 
x_0^{(N+1)r-j}\widetilde{M}^{j-Nr}\widetilde{M}^{(N-1)r}\widetilde{I}^{r}
\subseteq x_0^{(N+1)r-j}\widetilde{M}^{Nr}\widetilde{I}^{j-Nr},$$
where the last inclusion comes from the fact that
$j\leq (N+1)r$, hence $j-Nr\leq r$. But
we can convert $r-(j-Nr)$ of the factors of $\widetilde{I}^r$ to $\widetilde{M}$,
giving $\widetilde{I}^{r}\subseteq \widetilde{M}^{r-(j-Nr)}\widetilde{I}^{j-Nr}$.

(b) Since $r(N+1)-N\geq rN-(N-1)$, we have 
$\widehat{I}^{(r(N+1)-N)}\subseteq \widehat{I}^{(rN-(N-1))}$. Thus it is enough to show
$\widehat{I}^{(rN-(N-1))}\subseteq \widehat{I}^r$.
As in (a), for $m=rN-N+1$, we have
\[
\begin{split}
\widehat{I}^{(m)}&=(x_0^m)+x_0^{m-1}\widetilde{I}+\cdots +x_0\widetilde{I}^{(m-1)}+ \widetilde{I}^{(m)}\\
&= x_0^r((x_0^{m-r})+\cdots+\widetilde{I}^{(m-r)}) + x_0^{r-1}\widetilde{I}^{(m-r+1)}+\cdots+\widetilde{I}^{(m)}\\
&=x_0^r((x_0^{(r-1)(N-1)})+\cdots+\widetilde{I}^{((r-1)(N-1))}) + x_0^{r-1}\widetilde{I}^{((r-1)(N-1)+1)}+\cdots+\widetilde{I}^{(rN-N+1)}
\end{split}
\]
and
$\widehat{I}^r=(x_0)^r+x_0^{r-1}\widetilde{I}+\cdots+x_0\widetilde{I}^{r-1}
+\widetilde{I}^r$. Thus it suffices to show that 
$x_0^{r-j}\widetilde{I}^{((r-1)(N-1)+j)}\subseteq x_0^{r-j}\widetilde{I}^j$ for $1\leq j\leq r$.
But $x_0^{r-j}\widetilde{I}^{((r-1)(N-1)+j)}\subseteq x_0^{r-j}\widetilde{I}^j$
holds if $\widetilde{I}^{((r-1)(N-1)+j)}\subseteq \widetilde{I}^j$ does, which holds
if $I^{((r-1)(N-1)+j)}\subseteq I^j$ does.
Thus we are done, since the latter does hold: $I^{(Nj-(N-1))}\subseteq I^j$ holds by assumption, and
$I^{((r-1)(N-1)+j)}\subseteq I^{(Nj-(N-1))}$ holds since $(r-1)(N-1)+j\geq Nj-(N-1)$ for $1\leq j\leq r$.

(c) The proof for this part is essentially the same as for part (a) and so is left to the reader (but
it is included explicitly in the arXiv posting of this paper).

\end{proof}

As a corollary of the preceding result, we have the following theorem.

\begin{thm}\label{hypcor1} Assume $K$ has characteristic 0.
Let $Z=P_1+\cdots +P_n$ be a star configuration of points in $\pr N$, where we regard
$\pr N$ as a linear subspace of a larger dimensional projective space $\pr d$.
Let $I=I(Z_{\pr d})$. Then for all $r\geq1$ we have:
\begin{itemize}
\item[(a)] $I^{(rd)}\subseteq M^{r(d-1)}I^r$ and
\item[(b)] $I^{(rd-(d-1))}\subseteq M^{(r-1)(d-1)}I^r$.
\end{itemize} 
\end{thm}

\begin{proof}
(a) As noted in Section \ref{stars},
Conjecture \ref{fatptconj} holds for star configurations of points, hence
the result follows by Proposition \ref{hypprop}(a).

(b) Also as noted in Section \ref{stars},
Conjecture \ref{EvoEssenconj} holds for star configurations of points, hence
the result follows by Proposition \ref{hypprop}(c).
\end{proof}

As another corollary we have the following result. Part (a) was proved in
\cite{refDu} using different methods.

\begin{thm}\label{hypcor2} Assume $K$ has characteristic 0.
Let $Z=P_1+\cdots +P_n$ for $n\leq d+1$ general points of $\pr d$.
Let $I=I(Z)$. Then for all $r\geq1$ we have:
\begin{itemize}
\item[(a)] $I^{(rd)}\subseteq M^{r(d-1)}I^r$ and
\item[(b)] $I^{(rd-(d-1))}\subseteq M^{(r-1)(d-1)}I^r$.
\end{itemize} 
\end{thm}

\begin{proof}
Just note that $Z$ is a star configuration of points in $\pr{n-1}$ and apply
Theorem \ref{hypcor1}.
\end{proof}

\section{General points in the plane}\label{general}

Our focus here is for general points in the plane.
(Apart from \cite{refDu}, which shows that Conjecture \ref{fatptconj}
holds for finite sets of general points in $\pr3$, little so far
is known for general points in $\pr{N}$ for $N>2$.)

Let $I$ be the radical ideal of $n$ general points in $\pr 2$.
If $n$ is 1, 2 or 4, then Conjectures \ref{fatptconj} through \ref{ourconj} 
hold since the points comprise a complete intersection.
If $n=3$ or 5, Conjectures \ref{fatptconj} through \ref{ourconj}
hold since the points comprise a star configuration (in case $n=3$)
or lie on a smooth conic (if $n=5$), although
for Conjectures \ref{refinedChud}, \ref{refinedEV}, \ref{refinedEV2} and \ref{ourconj}
our proof above for both $n=3$ and 5 assumes that the characteristic is 0.

So now assume that $n\geq 6$ and $N=2$.

\ \newline\noindent Conjecture \ref{fatptconj}, that $I^{(rN)}\subseteq M^{r(N-1)}I^r$, holds by \cite[Corollary 3.10]{refHaHu}.

\ \newline\noindent Conjecture \ref{Essenconj}, that $I^{(rN-(N-1))}\subseteq I^r$, holds by \cite[Remark 4.3]{refBH2}, since $(rN-(N-1))/r=2-(1/r)\geq 3/2$ except in the trivial case that $r=1$.

\ \newline\noindent Conjecture \ref{p2conj}, that $I^{(m)}\subseteq I^r$ if $m/r\ge 2\alpha(I)/(\alpha(I)+1)$,
holds (see the discussion after Conjecture 4.1.4 of \cite{refHaHu}).

\ \newline\noindent Conjecture \ref{EvoEssenconj}, that $I^{(rN-(N-1))}\subseteq M^{(r-1)(N-1)}I^r$, holds by
\cite[Corollary 4.1.13]{refHaHu} but in some cases the proof assumes $\operatorname{char}(K)=0$.

\ \newline\noindent Conjecture \ref{EvoEssenconj2}, that $\alpha(I^{(rN-(N-1))})\ge r\alpha(I)+(r-1)(N-1)$, 
follows from Conjecture \ref{EvoEssenconj}.

\ \newline\noindent Conjecture \ref{refinedChud}, that $\frac{\alpha(I^{(m)})+N-1}{m+N-1}\leq \frac{\alpha(I^{(r)})}{r}$, and
Conjecture \ref{refinedEV}, that $I^{(t(m+N-1))}\subseteq M^t(I^{(m)})^t$, follow from
Conjecture \ref{refinedEV2}, that $I^{(t(m+N-1))}\subseteq M^{t(N-1)}(I^{(m)})^t$. 
We now consider some special cases of the latter, and of
Conjecture \ref{ourconj}, that  $I^{(t(m+N-1)-N+1)}\subseteq (I^{(m)})^t$
and $I^{(t(m+N-1)-N+1)}\subseteq M^{(t-1)(N-1)}(I^{(m)})^t$,
for $N=2$. 

\begin{prop}\label{squareprop}
Let $m,t\geq1$ and
let $I$ be the radical ideal of $n$ general points in $\pr2$ where $n = s^2$ for $s\geq 3$.  
Then $I^{(t(m+1))}\subseteq M^t(I^{(m)})^t$, $I^{(t(m+1)-1)}\subseteq (I^{(m)})^t$
and $I^{(t(m+1)-1)}\subseteq M^{t-1}(I^{(m)})^t$.
\end{prop}

\begin{proof} 
Let $n=9$. Then $\alpha(I^{(r)})=3r$ and
$\operatorname{reg}(I^{(r)})=3r+1$ for $r\geq 1$ \cite{refHa1}. 
Thus $\alpha(I^{(t(m+1))})=3t(m+1)\geq 
3tm+2t=t\operatorname{reg}(I^{(m)})+t$
holds so by Lemma \ref{genlemma} we have $I^{(t(m+1))}\subseteq M^t(I^{(m)})^t$,
which verifies Conjecture \ref{refinedEV2}.
Also, $\alpha(I^{(t(m+1)-1)})=3tm+3t-3$, and 
$\operatorname{reg}((I^{(m)})^t)\leq t(3m+1)$ by \cite{refGGP}.
Consider $I^{(t(m+1)-1)}\subseteq (I^{(m)})^t$.
If $t=1$, we have $I^{(t(m+1)-1)}=I^{(m)}= (I^{(m)})^t$, so assume $t>1$.
For $i<3tm+3t-3$, we have 
$(I^{(t(m+1)-1)})_i=0\subseteq ((I^{(m)})^t)_i$, while for $i\geq 3tm+3t-3$, we have
$i\geq 3tm+3t-3\geq t(3m+1)\geq \operatorname{reg}((I^{(m)})^t)$, so
$(I^{(t(m+1)-1)})_i\subseteq (I^{(mt)})_i=((I^{(m)})^t)_i$. Thus 
$I^{(t(m+1)-1)}\subseteq (I^{(m)})^t$ holds and
$I^{(t(m+1)-1)}\subseteq M^{t-1}(I^{(m)})^t$ now follows from
Lemma \ref{genlemma}(b).

For $s>3$, the proof is similar except we use $\alpha(I^{(r)})>rs$ \cite{refN}
and $\operatorname{reg}(I^{(r)})\leq s(r+\frac{1}{2})$
\cite[Lemma 2.5]{refHHF}. For example,
$t(m+1)s\geq ts(m+\frac{1}{2})+t$ for $s\geq 2$, so
we have $\alpha(I^{(t(m+1))})>t(m+1)s\geq 
ts(m+\frac{1}{2})+t\geq 
t\operatorname{reg}(I^{(m)})+t$.
Thus $I^{(t(m+1))}\subseteq M^t(I^{(m)})^t$ holds by Lemma \ref{genlemma}.
Similarly, we conclude $I^{(t(m+1)-1)}\subseteq (I^{(m)})^t$
and $I^{(t(m+1)-1)}\subseteq M^{t-1}(I^{(m)})^t$.
\end{proof}

\begin{rem} For $n>9$ if we assume the well-known SHGH Conjecture
(that $\dim_K(I^{(m)})_t=\max(0,\binom{t+2}{2}-n\binom{m+1}{2})$,
then we get $\alpha(I^{(t(m+1))})\geq t(m+1)\sqrt{n}$
and $\operatorname{reg}(I^{(m)})<(m+(1/2))\sqrt{n}+1$. This 
implies $I^{(t(m+1))}\subseteq M^t(I^{(m)})^t$ holds
if $t((m+(1/2))\sqrt{n}+1)+t\leq t(m+1)\sqrt{n}$, and this is true for $n\geq 16$.
\end{rem}

For $n\geq 10$ generic points, using known estimates for $\alpha(I^{(m)})$ and $\operatorname{reg}(I^{(m)})$
we can verify Conjecture \ref{refinedEV2} for all $t$ for $1\leq m\leq \sqrt{n+1}-1$.

\begin{prop}
Let $I$ be the radical ideal of $n$ generic points in $\pr2$ where $n \geq 10$.  
Then $I^{(t(m+1))}\subseteq M^t(I^{(m)})^t$ for all $t$ for $1\leq m\leq \sqrt{n+1}-1$.
\end{prop}

\begin{proof}
By Proposition \ref{squareprop} we may as well assume that $n$ is not a square.  
Given a finite set $S\subset \pr 2$ of points, we recall that the Seshadri constant 
$\varepsilon(S)$ is the least real number 
such that $\alpha(I)\geq \varepsilon(S)(\sum_{P\in S}m_P)$
for all nonnegative integers $m_P$ (not all zero) for the ideal $I=\cap_{P\in S}I(P)^{m_P}$
(see for example \cite[Remark 8.2.3]{refB. et al}).
The Seshadri constant 
$\varepsilon(n)$ for $n$ generic points in the plane satisfies $\varepsilon(n)\geq1/\sqrt{n+1}$ by
\cite[Proposition I.2(c)]{refH2} and \cite[Proposition I.3]{refH2}.
This means that $\alpha(I^{t(m+1)})\geq t(m+1)n/\sqrt{n+1}$.
By \cite[Corollary V.2.2(b)]{refH2}, for $n>9$ not a square we 
have $\operatorname{reg}(I^{(m)})\leq \frac{m+1}{\varepsilon(n)}-2\leq (m+1)\sqrt{n+1}-2$.
But $t(m+1)n/\sqrt{n+1}\geq t((m+1)\sqrt{n+1}-2)+t$ holds for $1\leq m\leq \sqrt{n+1}-1$. 
Thus by Lemma \ref{genlemma} we have
$I^{(t(m+1))}\subseteq M^t(I^{(m)})^t$ for $1\leq m\leq \sqrt{n+1}-1$.
\end{proof}

\section{Additional results}\label{additional}
It's possible to give a partial verification of Conjecture \ref{refinedChud}.

\begin{prop}
Let $I$ be the radical ideal of a finite set of points in $\pr N$. Assume $\operatorname{char}(K)=0$.
\begin{itemize}
\item[(1)] If $r\le N$, then
$$\frac{\alpha(I)+N-1}{N}\leq\frac{\alpha(I^{(r)})}{r}.$$

\item[(2)] If $m\le r\le N$, then
$$\frac{\alpha(I^{(m)})+N-1}{m+N-1}\leq\frac{\alpha(I^{(r)})}{r}.$$

\item[(3)] If either $\alpha(I^{(m)})=m$, or $m\leq r$ and $r-m<N$, then
$$\frac{\alpha(I^{(m)})+N-1}{m+N-1}\leq\frac{\alpha(I^{(r)})}{r}.$$

\end{itemize}
\end{prop}

\begin{proof} 
Note that (1) is a special case of (2).
For (2), assume $(\alpha(I^{(m)})+N-1)/(m+N-1)>\alpha(I^{(r)})/r$.
We know $\alpha(I^{(r)})\ge \alpha(I^{(m)})+r-m$, by repeated application of Lemma \ref{EulerLemma}.
Thus we have $(\alpha(I^{(m)})+N-1)/(m+N-1)>(\alpha(I^{(m)})+r-m)/r$.
This simplifies to $(r-m)(\alpha(I^{(m)})-m)>(N-1)(\alpha(I^{(m)})-m)$,
which is impossible.

(3) If $\alpha(I^{(m)})=m$, then we must show $1\leq \alpha(I^{(r)})/r$, but this holds since 
$\alpha(I^{(r)})\geq r$. If $m\leq r$ and $r-m<N$, then the proof is the same as the proof of (2),
since $\alpha(I^{(m)})-m\geq0$. 
\end{proof}

Conjectures \ref{refinedEV} and \ref{refinedEV2} are much stronger
than what is currently known. The best general result
known along these lines for the radical ideal $I$ of a finite set of points in $\pr N$ is that
$$I^{(t(N+m-1)+1)}\subseteq M(I^{(m)})^t$$
for all $t,m\geq1$; this follows from \cite[Theorem 0.1(1)]{refTY}.
Here we strengthen this result if $\operatorname{char}(K)=0$.

\begin{prop}
Let $I$ be the radical ideal of a finite set of points in $\pr N$. Assume $\operatorname{char}(K)=0$.
Then
$$I^{(t(N+m-1)+s)}\subseteq M^s(I^{(m)})^t$$
for all $s,t,m\geq1$. 
\end{prop}

\begin{proof} 
By Lemma \ref{EulerLemma} we have 
$I^{(Nt+(m-1)t+s)}\subseteq M I^{(Nt+(m-1)t+s-1)}$. Repeating and using 
Theorem \ref{HoHuThm} that
$I^{(Nt+(m-1)t)}\subseteq (I^{(m)})^t$, gives
$I^{(Nt+(m-1)t+s)}\subseteq M^s I^{(Nt+(m-1)t)}\subseteq M^s (I^{(m)})^t$.
\end{proof}

Finally, Conjecture \ref{EvoEssenconj2} 
has been verified in separate joint work by the second 
author for various special configurations
of points in ${\bf P}^2$ in characteristic 0
(the configurations in question consist of certain unions of sets of collinear points 
called line count configurations); see \cite{refCH}.

\end{document}